\documentclass{amsart}
\usepackage{graphicx}
\usepackage{amscd}
\newtheorem{theorem}{Theorem}[section]
\newtheorem{lemma}[theorem]{Lemma}

\newtheorem{problem}[theorem]{Problem}
\newtheorem{example}[theorem]{Example}
\newtheorem{corollary}[theorem]{Corollary}

\theoremstyle{definition}
\newtheorem{definition}[theorem]{Definition}
\newtheorem{remark}[theorem]{Remark}

\begin{document}

\title[Unknotting submanifolds of the 3-sphere by twistings]{Unknotting submanifolds of the 3-sphere by twistings}

\author{Makoto Ozawa}
\address{Department of Natural Sciences, Faculty of Arts and Sciences, Komazawa University, 1-23-1 Komazawa, Setagaya-ku, Tokyo, 154-8525, Japan}
\email{w3c@komazawa-u.ac.jp}
\thanks{The author is partially supported by Grant-in-Aid for Scientific Research (C) (No. 26400097), The Ministry of Education, Culture, Sports, Science and Technology, Japan}



\begin{abstract}
By the Fox's re-embedding theorem, any compact submanifold of the 3-sphere can be re-embedded in the 3-sphere so that it is unknotted.
It is unknown whether the Fox's re-embedding can be replaced with twistings.
In this paper, we will show that any closed 2-manifold embedded in the 3-sphere can be unknotted by twistings.
In spite of this phenomenon, we show that there exists a compact 3-submanifold of the 3-sphere which cannot be unknotted by twistings.
This shows that the Fox's re-embedding cannot always be replaced with twistings.
\end{abstract}

\maketitle

\section{Introduction}

Throughout this paper, we will work in the piecewise linear category.
We assume that a surface is a compact, connected 2-manifold and that a 2-manifold is possibly disconnected.

\begin{definition}
Let $X$ be a compact submanifold of the 3-sphere $S^3$.
Take a loop $C$ in $S^3-X$ which is the trivial knot in $S^3$.
By a $\pm 1$-Dehn surgery along $C$, we obtain another submanifold $X'$ of $S^3$ and call this operation a {\em twisting along} $C$, which is denoted by $(S^3,X)\overset{\mathrm{C}}{\longrightarrow} (S^3,X')$.

We note that $X'$ is homeomorphic to $X$, but the exterior of $X$, say $Y$, is usually not homeomorphic to one of $X'$, say $Y'$.
We also denote this deformation by $(S^3,Y)\overset{\mathrm{C}}{\longrightarrow} (S^3,Y')$.
\end{definition}


\begin{definition}
Let $X$ be a compact submanifold of $S^3$ which has $n$ connected components $X_1,\ldots,X_n$.
We say that $X=X_1\cup\cdots\cup X_n$ is {\em completely splittable} in $S^3$ if there exist $n-1$ mutually disjoint 2-spheres $S_1,\ldots,S_{n-1}$ in $S^3-X$ such that each connected component of $S^3-(S_1\cup\cdots\cup S_{n-1})$ contains a connected component of $X$.

When $X$ is completely splittable in $S^3$, if we cut open $S^3$ along $S_1\cup\cdots\cup S_{n-1}$ and glue $2(n-1)$ 3-balls along their boundaries, then we obtain $n$ pairs of the 3-sphere and the submanifold $(S^3,X_1),\ldots,(S^3,X_n)$.

For a connected component $X_i$ of $X$, we say that a pair $(S^3,X_i)$ is {\em unknotted} in $S^3$ if the exterior $E(X_i)=S^3-int N(X_i)$ consists of handlebodies.
We say that $X$ is {\em unknotted} if $X$ is completely splittable and for every pair $(S^3,X_i)$, $X_i$ is unknotted in $S^3$.
\end{definition}

\begin{remark}\label{re-embedding}
We remark that by the Fox's re-embedding theorem \cite{F}, any compact submanifold $M$ of $S^3$ can be re-embedded in $S^3$ so that $M$ is unknotted.
\end{remark}

The following is the main subject of this paper.

\begin{problem}\label{replace}
Can any Fox's re-embedding be replaced with twistings?
\end{problem}

It is well-known that Problem \ref{replace} is true for any closed 1-manifold and for any closed 2-manifold which bounds handlebodies.
In this paper, we will show that any closed 2-manifold embedded in the 3-sphere can be unknotted by twistings (Theorem \ref{main}).
In spite of this phenomenon, we show that there exists a compact 3-submanifold of the 3-sphere which cannot be unknotted by twistings (Corollary \ref{exist}).
This shows that the Fox's re-embedding cannot always be replaced with twistings.

\section{Main results}

\begin{definition}\label{compressible}
Let $F$ be a closed 2-manifold and $\alpha$ be a loop, namely, simple closed curve  in $F$.
We say that $\alpha$ is {\em inessential} in $F$ if it bounds a disk in $F$.
Otherwise, $\alpha$ is {\em essential}.
We define the {\em breadth} $b(F)$ of $F$ as the maximal number of mutually disjoint, mutually non-parallel essential loops in $F$.

Let $F$ be a closed 2-manifold embedded in $S^3$ with $b(F)>0$.
We say that $F$ is {\em compressible} in $S^3$ if there exists a disk $D$ embedded in $S^3$ such that $D\cap F=\partial D$ and $\partial D$ is essential in $F$.
Such a disk is called a {\em compressing disk} for $F$.
Then by cutting $F$ along $\partial D$, and pasting two parallel copies of $D$ to its boundaries, we obtain another closed 2-manifold $F'$ with $b(F')<b(F)$.
Such an operation is called a {\em compression along} $D$.
Conversely, if $F'$ is obtained from $F$ by a compression along $D$, then there exists a dual arc $\alpha$ with respect to $D$, that is, $\alpha$ intersects $D$ in one point and $\alpha\cap F'=\partial \alpha$ such that $F$ can be recovered from $F'$ by tubing along $\alpha$. See Figure \ref{compression}.
\begin{figure}[htbp]
	\begin{center}
	\includegraphics[trim=0mm 0mm 0mm 0mm, width=.6\linewidth]{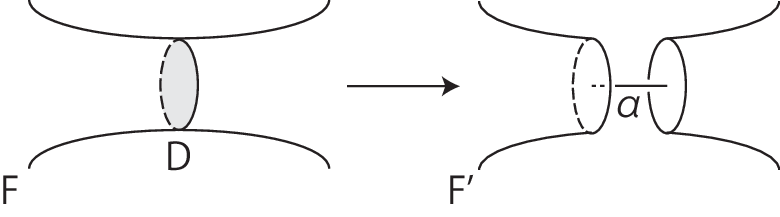}	
	\end{center}
	\caption{Compression along $D$}
	\label{compression}
\end{figure}
\end{definition}

\begin{remark}\label{key}
We remark that if $b(F)=0$, then $F$ consists of only 2-spheres and by the Alexander's theorem \cite{A}, $F$ is unknotted in $S^3$.
On the other hand, we remark that if $b(F)>0$, then by \cite{F} or \cite{H}, $F$ is {\em compressible} in $S^3$.
\end{remark}

\begin{theorem}\label{main}
Any closed 2-manifold embedded in the 3-sphere can be unknotted by twistings.
\end{theorem}

\begin{proof}
Let $F$ be a closed 2-manifold consisting of $n$ closed surfaces $F_1,\ldots,F_n$ embedded in $S^3$.
We will prove Theorem \ref{main} by induction on the breadth $b(F)$.

In the case of $b(F)=0$, by Remark \ref{key}, $F$ is unknotted.

Next suppose that when $b(F)<b$, Theorem \ref{main} holds, and assume that $b(F)=b$.
Then by Remark \ref{key}, there exists a compressing disk $D$ for $F$.
Let $F'$ be the closed 2-manifold obtained from $F$ by a compression along $D$.
Then there exists an arc $\alpha$ such that $\alpha$ intersects $D$ in one point, $\alpha\cap F'=\partial \alpha$, and $F$ can be obtained from $F'$ by tubing along $\alpha$.
Since $b(F')<b$, by the assumption of induction, $F'$ can be unknotted by twistings.
Thus there exists a sequence of twistings 
$$(S^3,F')\overset{\mathrm{C_1}}{\longrightarrow} (S^3,F'^{(1)}) \overset{\mathrm{C_2}}{\longrightarrow}  \cdots \overset{\mathrm{C_m}}{\longrightarrow} (S^3,F'^{(m)}),$$ where $F'^{(m)}$ is unknotted.
In each stage, we may assume that $C_i\cap \alpha=\emptyset$ for $i=1,\ldots,m$.
Therefore, this sequence extends to a sequence of twistings $$(S^3,F)\overset{\mathrm{C_1}}{\longrightarrow} (S^3,F^{(1)}) \overset{\mathrm{C_2}}{\longrightarrow}  \cdots \overset{\mathrm{C_m}}{\longrightarrow} (S^3,F^{(m)}).$$

Let $R$ be the closure of a connected component of $S^3-F'$ which contains $\alpha$, and put $\partial R=F'_1\cup\cdots\cup F'_k$, where $F'_1,\ldots,F'_k$ are connected components of $F'^{(m)}$.
Since $F'^{(m)}$ is unknotted, $F'_1\cup\cdots\cup F'_k$ bounds $k$ handlebodies $V_1,\ldots,V_k$ in $S^3-int R$, and $V_1\cup\cdots\cup V_k$ is unknotted in $S^3$, namely, $V_1\cup\cdots\cup V_k$ is ambient isotopic to a regular neighborhood of a plane graph $G$ on the 2-sphere $S$.
Then by crossing changes on $\alpha$ and crossing changes between $\alpha$ and $V_i$, $\alpha$ can be unknotted, that is, $\alpha$ is isotopic to an arc on $S$ as shown in Figure \ref{unknotting}.
\begin{figure}[htbp]
	\begin{center}
	\includegraphics[trim=0mm 0mm 0mm 0mm, width=.7\linewidth]{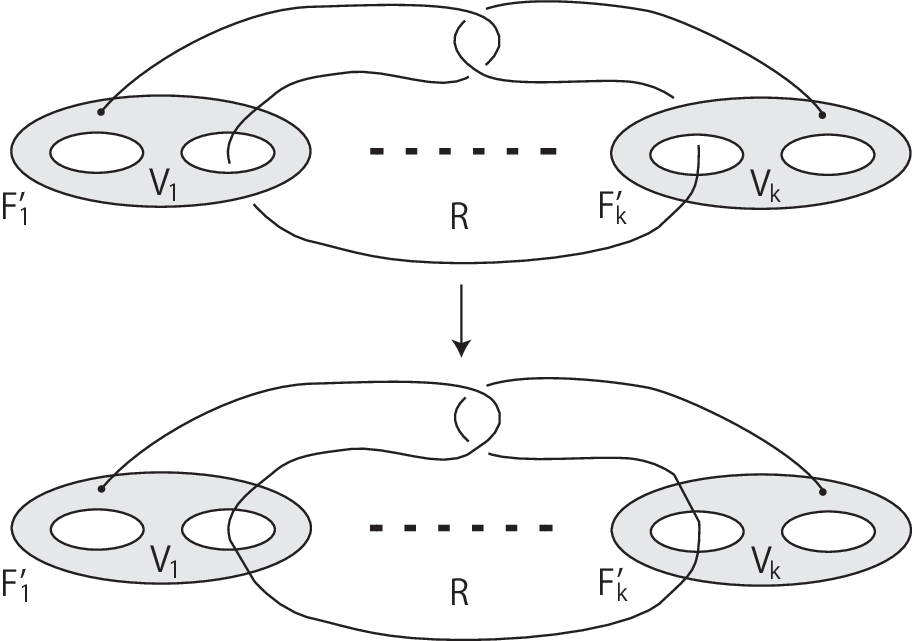}	
	\end{center}
	\caption{Unknotting $\alpha$ in $R$}
	\label{unknotting}
\end{figure}
Since these crossing changes are obtained by twistings, there is a sequence of twistings $$(S^3,F'^{(m)})\overset{\mathrm{C_{m+1}}}{\longrightarrow} (S^3,F'^{(m+1)}) \overset{\mathrm{C_{m+2}}}{\longrightarrow} \cdots \overset{\mathrm{C_{m+l}}}{\longrightarrow} (S^3,F'^{(m+l)}),$$ where $F'^{(m)}, F'^{(m+1)}, \ldots, F'^{(m+l)}$ are equivalent and $C_{m+i}\cap \alpha =\emptyset$ for $i=1,\ldots,l$.
Therefore, this sequence extends to a sequence of twistings $$(S^3,F^{(m)})\overset{\mathrm{C_{m+1}}}{\longrightarrow} (S^3,F^{(m+1)}) \overset{\mathrm{C_{m+2}}}{\longrightarrow} \cdots \overset{\mathrm{C_{m+l}}}{\longrightarrow} (S^3,F^{(m+l)}).$$
Hence, by tubing $F'$ along $\alpha$, we obtain a sequence of twistings $$(S^3,F)\overset{\mathrm{C_1}}{\longrightarrow} (S^3,F^{(1)}) \overset{\mathrm{C_2}}{\longrightarrow}  \cdots \overset{\mathrm{C_m}}{\longrightarrow} (S^3,F^{(m)})$$
$$\overset{\mathrm{C_{m+1}}}{\longrightarrow} (S^3,F^{(m+1)}) \overset{\mathrm{C_{m+2}}}{\longrightarrow} \cdots \overset{\mathrm{C_{m+l}}}{\longrightarrow} (S^3,F^{(m+l)}),$$
where $F^{(m+l)}$ is unknotted.
\end{proof}



\begin{example}
We recall an example of closed surface $H$ of genus 2 given by Homma \cite{H}, see also \cite[4.1 Theorem]{S} as shown in Figure \ref{homma}.
The surface $H$ separates $S^3$ into two components $W_1$ and $W_2$, where $W_1$ is homeomorphic to the exterior of the 4-crossing Handcuff graph $4_1$ in the table of \cite{IKMS}, and $W_2$ is a boundary connected sum of two trefoil knot exteriors.
It is remarkable that $H$ is incompressible in $W_1$, whereas $H$ has only one compressing disk $D$ in $W_2$ up to isotopy by \cite{T}, \cite{S2}.
\begin{figure}[htbp]
	\begin{center}
	\includegraphics[trim=0mm 0mm 0mm 0mm, width=.8\linewidth]{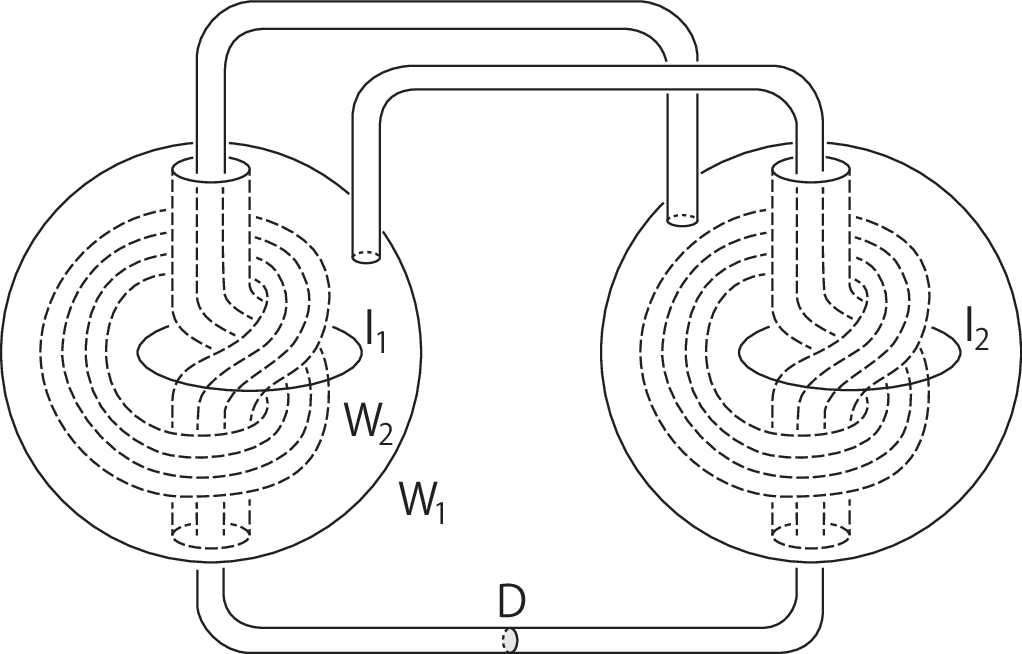}	
	\end{center}
	\caption{The Homma's closed surface}
	\label{homma}
\end{figure}
\begin{figure}[htbp]
	\begin{center}
	\includegraphics[trim=0mm 0mm 0mm 0mm, width=.6\linewidth]{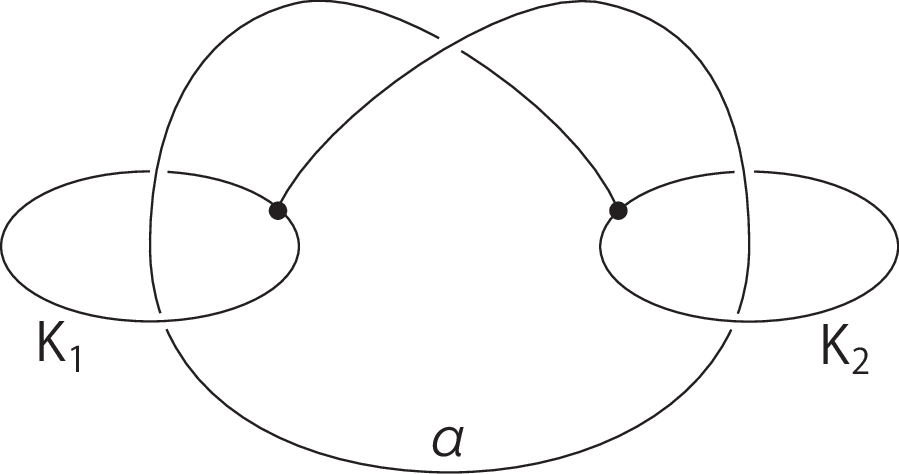}	
	\end{center}
	\caption{The 4-crossing Handcuff graph $4_1$}
	\label{handcuff}
\end{figure}
\end{example}

In spite of Theorem \ref{main}, there is a following phenomenon.

\begin{theorem}\label{main2}
The Homma's surface $H$ cannot be unknotted by twistings in $W_1$.
\end{theorem}

\begin{proof}
Suppose that there exists a a sequence of twistings 
$$(S^3,W_2)\overset{\mathrm{C_1}}{\longrightarrow} (S^3,W_2^{(1)}) \overset{\mathrm{C_2}}{\longrightarrow}  \cdots \overset{\mathrm{C_n}}{\longrightarrow} (S^3,W_2^{(n)}),$$ where each $C_i$ is contained in $S^3-W_2^{(i-1)}$ and $W_2^{(n)}$ is unknotted.

We regard $W_2$ as $E_1\cup N(\alpha)\cup E_2$, where $E_1$ and $E_2$ are two trefoil knot exteriors and $N(\alpha)$ is a 1-handle along a dual arc $\alpha$ with respect to $D$.

\begin{lemma}\label{lemma1}
For any twisting $(S^3,W_2)\overset{\mathrm{C}}{\longrightarrow} (S^3,W_2')$, there exists a disk $\Delta$ in $S^3$ with $\partial \Delta=C$ such that $\Delta\cap(E_1\cup E_2)=\emptyset$.
\end{lemma}

\begin{proof}[Proof of Lemma \ref{lemma1}]
Since the exterior $E(C)=S^3-int N(C)$ of $C$ is the solid torus, both of $\partial E_1$ and $\partial E_2$ are compressible in $E(C)-int(E_1\cup E_2)$.
Therefore, there exists a compressing disk $\Delta$ for $\partial E(C)$ in $E(C)$ such that $\Delta\cap (E_1\cup E_2)=\emptyset$.
This disk $\Delta$ can be extended to a disk bounded by $C$.
\end{proof}

By Lemma \ref{lemma1}, we may assume that $\alpha$ intersects $\Delta$ transversely and conclude that any twisting along $C$ takes effect only on $\alpha$.

Let $l_i$ be a loop in $E_i$, which is the trivial knot in $S^3$, such that the solid torus $V_i$ is obtained from $E_i$ by a twisting along $l_i$ as shown in Figure \ref{homma}.
Put $H_2=V_1\cup N(\alpha)\cup V_2$.
Thus we have $(S^3,W_2)\overset{\mathrm{l_1\cup l_2}}{\longrightarrow} (S^3,H_2)$.

\begin{lemma}\label{lemma2}
For $i=1,2$, there exists a disk $\delta_i$ in $S^3$ with $\partial \delta_i=l_i$ and $\delta_1\cap\delta_2=\emptyset$ such that $\delta_i\cap \Delta=\emptyset$.
\end{lemma}

\begin{proof}[Proof of Lemma \ref{lemma2}]
By Lemma \ref{lemma1}, the 3-submanifold $N(\Delta)\cup E_1\cup E_2$ is completely splittable in $S^3$.
Therefore, there exists a disk $\delta_i$ $(i=1,2)$ bounded by $l_i$ such that $\delta_1\cap\delta_2=\emptyset$ and $\delta_i\cap \Delta=\emptyset$.
\end{proof}

By Lemma \ref{lemma2}, we have the following lemma.

\begin{lemma}\label{lemma3}
The following diagram is commutative.
$$
\begin{CD}
    (S^3,W_2)  @>C>> (S^3,W_2')  \\
    @Vl_1\cup l_2VV    @Vl_1\cup l_2VV \\
    (S^3,H_2) @>C>> (S^3,H_2') 
 \end{CD}
 $$
\end{lemma}

By the supposition and Lemma \ref{lemma3}, we have the following commutative diagram.
$$
\begin{CD}
    (S^3,W_2)  @>C_1>> (S^3,W_2^{(1)}) @>C_2>> \cdots @>C_n>> (S^3,W_2^{(n)}) \\
    @Vl_1\cup l_2VV  @Vl_1\cup l_2VV  @.  @Vl_1\cup l_2VV \\
    (S^3,H_2) @>C_1>> (S^3,H_2^{(1)}) @>C_2>> \cdots @>C_n>> (S^3,H_2^{(n)}) 
 \end{CD}
 $$

Since $W_2^{(n)}$ is unknotted in $S^3$, $H_2^{(n)}$ is also unknotted in $S^3$.
It follows from \cite{ST} or \cite{OT} that the Handcuff graph corresponding to $H_2^{(n)}$ is trivial.
Thus, the Handcuff graph $4_1$ corresponding to $H_2$ can be trivialized by crossing changes only on $\alpha$.
However, it contradicts the following lemma. 
\end{proof}





\begin{lemma}\label{lemma4}
The Handcuff graph $4_1$ cannot be trivialized by crossing changes only on its cut edge.
\end{lemma}

\begin{proof}[Proof of Lemma \ref{lemma4}]
Let $K_1\cup\alpha\cup K_2$ be the Handcuff graph $4_1$, whose exterior is homeomorphic to $W_1$.
We take a double branched cover of $S^3$ along the trivial link $K_1\cup K_2$ as follows.
Let $D_i$ be a disk bounded by $K_i$ which intersects $\alpha$ in one point $(i=1,2)$.
We cut open $S^3$ along $D_1\cup D_2$ and take a copy of it.
Those 3-manifolds are both homeomorphic to $S^2\times I$ and whose boundary consists of 2-spheres $D_1^+\cup D_1^-$, $D_2^+\cup D_2^-$, ${D'}_1^+\cup {D'}_1^-$, ${D'}_2^+\cup {D'}_2^-$.
Then by gluing $D_1^+$ and ${D'}_1^-$, $D_1^-$ and ${D'}_1^+$, $D_2^+$ and ${D'}_2^-$, $D_2^-$ and ${D'}_2^+$  we obtain $S^2\times S^1$ and a knot $\tilde{\alpha}$ obtained from $\alpha$ and $\alpha'$ as shown in Figure \ref{double}.
We note that $[\tilde{\alpha}]=3[\gamma]$ in $H_1(S^2\times S^1;\mathbb{Z})\cong\mathbb{Z}$, where $\gamma$ is a generator of $H_1(S^2\times S^1;\mathbb{Z})$.

\begin{figure}[htbp]
	\begin{center}
	\includegraphics[trim=0mm 0mm 0mm 0mm, width=.9\linewidth]{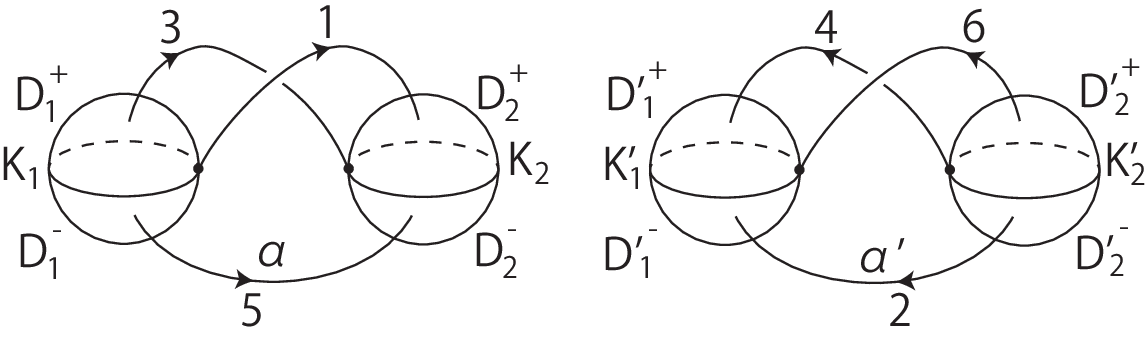}	
	\end{center}
	\caption{The double branched cover of $S^3$ along $K_1\cup K_2$}
	\label{double}
\end{figure}

Suppose that $K_1\cup\alpha\cup K_2$ is unknotted by crossing changes on $\alpha$.
Then the homology class $[\tilde{\alpha}]$ in $H_1(S^2\times S^1;\mathbb{Z})$ does not change by the crossing changes, and we have $[\tilde{\alpha}]=3[\gamma]$.
However, since $K_1\cup\alpha\cup K_2$ is trivial, we have $[\tilde{\alpha}]=[\gamma]$.
This is a contradiction.
\end{proof}

By the Fox's re-embedding theorem \cite{F}, there exists a re-embedding of $W_2$ in $S^3$ such that $W_2$ is unknotted.
However, this Fox's re-embedding cannot be obtained by twistings.

\begin{corollary}\label{exist}
There exists a 3-submanifold of $S^3$ which cannot be unknotted by twistings.
\end{corollary}

\begin{proof}
Take $W_2$ as a 3-submanifold of $S^3$.
\end{proof}

\section{Concluding remarks}

We conclude with some remarks on topics related with subjects in this paper.

\subsection{Fox's re-embeddings and Dehn surgeries}
We remark that by \cite[Theorem 1.6]{OS}, there exists a null-homologous link $L$ in $W_1$, which is reflexive in $S^3$, such that a handlebody can be obtained from $W_1$ by a $1/\mathbb{Z}$-Dehn surgery along $L$, that is, $[L]=0$ in $H_1(W_1;\mathbb{Z})$ and there exists a surgery slope $1/n_i$ for each component $L_i$ of $L$ such that a pair of $S^3$ and a handlebody is obtained from $(S^3,W_1)$ by a Dehn surgery along $L$. 
Therefore, the Fox's re-embedding can be replaced with a Dehn surgery along a link.
At the time of writing  of \cite{OS}, it was unknown whether this Dehn surgery can be replaced with twistings.
Corollary \ref{exist} shows that it is not always true.

\subsection{The number of equivalence classes by twistings}
Corollay \ref{exist} and the Fox's re-embedding theorem shows that there exists a compact 3-submanifold $W_2$ of $S^3$ such that the number of equivalence classes of $W_2$ by twistings is at least two.
It can be observed that along the proofs of Theorem \ref{main2} and Lemma \ref{lemma4}, $W_2$ has  infinitely many equivalence classes by twistings.
To see this, consider an embedding of $W_2=E_1\cup N(\alpha)\cup E_2$ in $S^3$, where $\alpha$ goes through $E_1$ and $E_2$ $n$ times respectively.
Thus, $\alpha$ intersects $D_i$ in $n$ points $(i=1,2)$.
Then we have that $[\tilde{\alpha}]=(2n+1)[\gamma]$ in $H_1(S^2\times S^1;\mathbb{Z})$ and this homology class is an invariant for crossing changes on $\alpha$ and hence twistings on $W_2$.
Therefore, by varying $n$, we obtain infinitely many equivalence classes of $W_2$ by twistings. 

\subsection{Nugatory twistings on submanifolds}
It is known as the Lin's nugatory crossing conjecture in {\cite[Problem 1.58]{K}} that if an oriented knot does not change by a crossing change, then the crossing is nugatory.
This conjecture holds on the trivial knot by \cite{ST1}, 2-bridge knots by \cite{T1} and fibered knots by \cite{K1}.
Analogously, we propose the "nugatory twisting conjecture" on submanifolds of $S^3$, that is, if a submanifold of $S^3$ does not change by a twisting, then the twisting is nugatory.

\subsection{Uniqueness of embeddings of submanifolds}
Any closed 1-manifold or closed orientable 2-manifold except for the 2-sphere has infinitely many non-equivalent embeddings in $S^3$, namely links or knotted surfaces.
However, it is well-known by \cite{GL} that any non-trivial knot exterior in $S^3$ has only one embedding in $S^3$.
We remark that any non-trivial knot exterior $X$ satisfies the following condition: any non-contractible loop $l$ in $S^3-X$ is non-trivial in $S^3$, that is, $X$ does not admit a non-trivial twisting.

In the below-mentioned, if such a condition is not satisfied, then there are infinitely many embeddings of a submanifold contrary to the case of non-trivial knot exteriors.
Let $X$ be a 3-submanifold $X$ of $S^3$.
Suppose that there exists a non-contractible loop $l$ in $S^3-X$ which is trivial in $S^3$.
Then, the exterior $E(l)=S^3-int N(l)$ of $l$ is a solid torus containing $X$.
By re-embedding of $E(l)$ in $S^3$ so that it is knotted in $S^3$, we obtain infinitely many embeddings of $X$ in $S^3$.
More generally, if $X$ is contained in a submanifold $Y$ so that $Y-X$ is irreducible and $Y$ has infinitely many embeddings in $S^3$, then one can obtain infinitely many embeddings of $X$ in $S^3$.

\bigskip

\noindent{\bf Acknowledgements.}
The author would like to thank Kouki Taniyama, Ryo Nikkuni and Yukihiro Tsutsumi for helpful comments.

After the publication of this paper, Arkadi Skopenkov, who is the reviewer of Mathematical Reviews of American Mathematical Society, pointed out several ambiguous definitions.
I would like to thank him and revised in arXiv:1609.06573v3.

\bibliographystyle{amsplain}

\end{document}